\documentclass[oneside,11pt]{article}
\usepackage[english]{babel}
\usepackage{verbatim}
\usepackage{graphicx}
\usepackage{syntonly}
\usepackage{amsfonts}
\usepackage{amssymb}
\usepackage{amsthm}
\usepackage{mathrsfs}
\usepackage[all]{xy}

\makeatletter

\def\NZQ{\mathbb}               
\def\NN{{\NZQ N}}

\def\RR{{\NZQ R}}

\theoremstyle{plain}
\newtheorem{theorem}{Theorem}[section]

\newtheorem{lemma}[theorem]{Lemma}
\newtheorem{proposition}[theorem]{Proposition}

\theoremstyle{definition}
\newtheorem{definition}[theorem]{Definition}
\newtheorem{example}[theorem]{Example}
\newtheorem{remark}[theorem]{Remark}

\newcommand{\eqnsection}{
\renewcommand{\theequation}{{\thesection.\arabic{equation}}} 
\renewcommand{\con}{\mathrm{conv}}

\makeatletter \csname @addtoreset\endcsname{equation}{section}
\makeatother} 

\begin{document}
\thispagestyle{empty}
\date{}

\title{\textbf{\large ALGORITHMIC RELEASES ON THE\\[1mm] SPANNING TREES OF SUITABLE GRAPHS\\[3mm]}}

\author{\textsc{Maurizio Imbesi}\\[2mm]
\small{\em Department of Mathematical and Computer Sciences, Physical and Earth Sciences}\\
\small{\em University of Messina, Viale F. Stagno d'Alcontres 31, I-98166 Messina, Italy}\\
\small{{\tt e-mail address:} maurizio.imbesi@unime.it}\\[6mm]
\textsc{Monica La\,Barbiera}\\[2mm]
\small{\em Department of Electrical, Electronic and Computer Engineering}\\
\small{\em University of Catania, Viale A. Doria 6, I-95125 Catania, Italy}\\
\small{{\tt e-mail address:} monica.labarbiera@unict.it}\\[6mm]
\textsc{Santo Saraceno}\\[2mm]
\small{\em Department of Mathematical and Computer Sciences, Physical and Earth Sciences}\\
\small{\em University of Messina, Viale F. Stagno d'Alcontres 31, I-98166 Messina, Italy}\\
\small{{\tt e-mail address:} snapmode91@gmail.com}
}

\maketitle
\noindent
\textsc{Abstract.} {\small In this paper algebraic and combinatorial properties and a computation of the number of the spanning trees are developed for certain graphs.\\ To this purpose, an original method, independent of the spectrum of the Laplacian matrix associated to the graph, is discussed. It represents an alternative process to compute how many and which are the spanning trees of any graph and substantially consists in joining the spanning trees on the ground of the number of common edges between the inner cycles of it.\\ The algorithm and its source code for determining the collection of all edge-sets of the spanning trees for the class of Jahangir graphs are displayed.\\ An application involving such graphs in order to get a satisfactory degree of security in transmitting confidential information is given, and finally symmetry properties of them are highlighted.}\\[4mm]
{\small 2020 \emph{Mathematics Subject Classification}. Primary 68R10; Secondary  05C05, 05C85.}\\
{\small \emph{Key Words and Phrases}. Combinatorics; graph theory; spanning trees.}

\section*{Introduction}
Let $G$ be a finite simple connected cyclic graph having vertex set $V(G)$ and edge set $E(G)$. We refer to \cite{{GR},{V}} for a detailed presentation of classical algebraic topics about graph theory. In \cite{{IL3},{I},{IL},{IL1},{IL2}} interesting results about algebraic and combinatorial properties linked to finite graphs can be found. A spanning tree of $G$ is an acyclic connected subgraph of $G$ that contains all the vertices of $G$. Let's denote by $s(G)$ the collection of all edge-sets of the spanning trees of $G$.
An effective analytical method for obtaining systematically all the existing spanning trees of $G$ is the so-called cutting-down method: it consists of removing an appropriate number of edges from the graph for making it acyclic.\par\noindent
The work is devoted in studying an alternative method for the computation of the spanning trees of simple connected cyclic graphs, including for instance those considered in \cite{{ARK},{ARK1}}. Specifically, combinatorial properties of the spanning trees for the class of Jahangir graphs, defined in \cite{LJM}, will be discussed and an algorithmic method to determine how many and what are the spanning trees of such type of graphs will be developed. This provides a general procedure for the calculation of the number $\sigma$ of their spanning trees.
The paper is structured as follows.\par
In Section 1 we introduce fundamental notions on graph theory like simple, connected, cyclic graphs, incidence and adjacency matrices and the Laplacian matrix associated to a graph, see also \cite{Bol}. Moreover, an important theorem, the Matrix Tree Theorem, due to G. Kirchhoff, is stated; the total number of the spanning trees of any graph follows from the nonzero eigenvalues of the associated Laplacian matrix (\cite{{B},{S}}).\par
In Section 2 we analyze Jahangir graphs $J_{n,m}$, namely graphs which consist of $m$ edges $\{v,v_1\},\dots,\{v,v_m\}$, together with $m$ paths, all of length $n$, having endpoints $v_1$,$v_2$; $v_2$,$v_3$; \dots; $v_{m-1}$,$v_m$; $v_m$,$v_1$. Their shape, similar to a flower with a center, was inspired by a drawing carved in the mausoleum of the Indian Grand Mogul Jahangir (1569-1627) located in Lahore, Pakistan.
To find the spanning trees of a Jahangir graph, it is introduced a method independent of the spectrum of the associated Laplacian matrix which substantially consists of joining the spanning trees on the ground of the number of common edges between the inner cycles of it. So, this represents an alternative process to compute how many and what are the spanning trees. The original algorithm and the source code for determining the collection of all edge-sets of the spanning trees for Jahangir graphs $J_{n,m}$ are displayed.\par
In Section 3 an application of sensitive data transmission arising from security real problems is illustrated.\par
In Section 4 we highlight symmetry properties of Jahangir graphs $J_{n,m}$ and consider interesting relationships on the number of spanning trees of those having either same $m$ or same $n$.
Precisely, for equal indices $m$, two distinct Jahangir graphs, with $n$ large enough, tend to have the same number of spanning trees; for equal indices $n$, we show that the ratio between the number of spanning trees of two Jahangir graphs with consecutive indices $m$, tends, with increasing $m$, to a constant $\delta_n$\,, distinct for each choice of $n$.

\medskip
\section{Preliminary notions and classical methods}
Give basic definitions and notations that will be used throughout the paper.\\
A \textit{(finite) graph} $G$ is an ordered triple $(V(G), E(G), f)$ such that $V(G)= \{v_1, \ldots, v_n\}$ is the set of the vertices of $G$, $E(G)=\{\{v_i,v_j\} \ | \ v_i, v_j \in V(G) \}$ the set of edges of $G$ and $f: E(G) \longrightarrow V(G)\times V(G)$ the incidence function.\\
$G$ is said to be \textit{simple} if, for all $\{v_i,v_j\}\in E(G)$, it is $v_i\neq v_j\,$. In other words, a simple graph is an undirect (or a non oriented) graph without loops.\\
A \textit{subgraph} of $G$ is a graph with all of its vertices and edges belonging to $G$.\\
A \textit{spanning subgraph} of $G$ is a subgraph containing all the vertices of $G$.\\
The \emph{degree} of a vertex $v$ of $G$ is the number of edges incident to $v$\,.\\
A \textit{walk} of $G$ of length $q$ is an alternating sequence of $q\!+\!1$ vertices and $q$ edges beginning and ending with vertices in which each edge is incident with the two vertices immediately preceding and following it.\\
A \textit{path} of $G$ is a walk having all the vertices, and thus all the edges, distinct.\\
A graph $G$ is called \textit{connected} if every pair of vertices are joined by a path.\\
A walk of $G$ is said to be \textit{closed} if the defining sequence begins and ends at the same vertex.\\
A closed path of $G$ of length $q\geq 3$ is called \textit{cycle}; in particular, it is a subgraph $C_q$ such that $E(C_q)=\big{\{}\{v_{i_1},
v_{i_2}\},\{v_{i_2}, v_{i_3}\}, \ldots,$ $\{v_{i_{q-1}}, v_{i_q}\},
\{v_{i_q}, v_{i_1}\} \big{\}}$, where $\{ v_{i_1}, \ldots, v_{i_q}\}
\in V(G)$ and $v_{i_j} \neq v_{i_k}$ if $i_j \neq i_k$.\\
A graph which has no cycles is called \textit{acyclic}.\\
A \textit{tree} is a connected acyclic graph. Any graph without cycles is a \textit{forest}, thus the connected subgraphs of a forest are trees.
\begin{definition}\rm{
A \textit{spanning tree} of a simple connected finite graph $G$ is a subtree of $G$ that contains every vertex of $G$.\\
We denote by $s(G)$ the collection of all edge-sets of the spanning trees of $G$:\vspace{-1.5mm}
$$s(G) = \{E(T_i) \subset E(G), \textrm{where} \  T_i \  \textrm{is a spanning tree of} \ G\}.$$
}
\end{definition}
\vspace{-1mm}
\noindent It is well-known that, for any simple finite connected graph, spanning trees always exist. One can systematically find a spanning tree by using the cutting-down method (see \cite{AW}), which in particular says that a spanning tree of a simple finite connected graph can be obtained by removing one edge from each cycle appearing in the graph.\\[.5mm]
We denote by $\sigma(G)$ the number of spanning trees of $G$.
\begin{example}\label{e1}\rm{
Let $G$ be the graph with $V(G)= \{v_1, v_2, v_3, v_4\}$ and
$E(G) = \{e_1, e_2, e_3, e_4\}$
\begin{figure}[htbp]
\begin{center}
   \includegraphics[scale=.85]{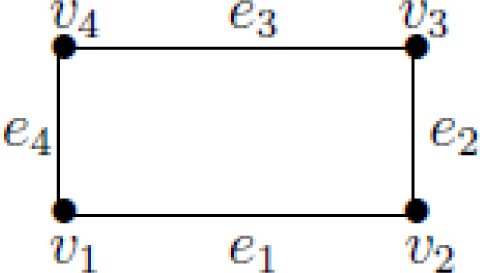}
\end{center}
\end{figure}
where $e_1= \{v_1,v_2\}$, $e_2= \{v_2,v_3\}$, $e_3= \{v_3,v_4\}$, $e_4= \{v_1,v_4\}.$\vspace{-5mm}
By using the cutting-down method for $G$ one obtains: \vspace{-1.5mm}
$$s(G) = \{\{e_2, e_3, e_4\}, \{e_1, e_3, e_4\}, \{e_1, e_2, e_4\}, \{e_1, e_2, e_3\}\}.$$
}
\end{example}
\begin{definition}\rm{
Let $G$ be a simple graph with vertex set $V(G)=\{v_1,\dots, v_n\}$ and edge set $E(G)=\{e_1,\dots, e_p\}$. We call \emph{incidence matrix} $M=(m_{ij}), \, i=1,\dots, n, \, j=1,\dots, p$, associated to $G$ the $n\times p$ matrix such that:\vspace{-1.2mm}
\begin{itemize}
\item $m_{ij} = 1$, if the vertex $\{v_i\}$ meets the edge $\{e_j\}$;\vspace{-1.8mm}
\item $m_{ij} = 0$, if the vertex $\{v_i\}$ is external to the edge $\{e_j\}$.
\end{itemize}
}
\end{definition}
\begin{definition}\rm{
Let $G$ be a graph with vertex set $V(G)=\{v_1,\dots, v_n\}$. We call \emph{adjacency matrix} $A=(a_{ij}), \, i,j=1,\dots, n$, associated to $G$ the $n\times n$ matrix such that:\vspace{-1.2mm}
\begin{itemize}
\item $a_{ij} = 1$, if $\{v_i,v_j\} \in  E(G)$\vspace{-1.8mm}
\item $a_{ij} = 0$, if $\{v_i,v_j\} \notin  E(G)$.
\end{itemize}
}
\end{definition}
\begin{definition}\rm{
Let $G$ be a graph with vertex set $V(G)=\{v_1,\dots, v_n\}$. We call \emph{matrix of degrees} $D=(d_{ij}), \, i,j=1,\dots, n$, associated to $G$ the $n\times n$ matrix such that:\vspace{-1.2mm}
\begin{itemize}
\item $d_{ij} = \deg (v_i)$, if $i=j$\vspace{-1.8mm}
\item $d_{ij} = 0$, if $i \neq j$.
\end{itemize}
}
\end{definition}
\begin{definition}\rm{
Let $G$ be a graph with vertex set $V(G)=\{v_1,\dots, v_n\}$. We call \emph{Laplacian matrix} associated to $G$ the matrix $L = D - A$. }
\end{definition}
\begin{remark}\rm{
If $M$ is the incidence matrix associated to a direct graph $G$, so that the nonzero entries in each column are given by $-1$ and $1$, then the Laplacian matrix $L$ associated to $G$ is the matrix $L=M\,^t\!M$\,.
}
\end{remark}
\noindent An effective theoretical method due to G. Kirchhoff
for determining the spanning trees of any suitable graph is provided in the following
\begin{theorem} [Kirchhoff's Matrix Tree Theorem]
Let $G$ be a connected simple graph with $n$ vertices and associated Laplacian matrix $L$.
If $\sigma(G)$ is the number of the spanning trees of $G$, then\vspace{-1,5mm}
$$\sigma (G) = \frac{\lambda_1\lambda_2 \cdots \lambda_{n-1}}{n},$$
where $\lambda_1,\lambda_2,  \ldots, \lambda_{n-1}$ are the nonzero eigenvalues of $L$.
\end{theorem}
\begin{proof}
An implicit version of the statement first appeared in \cite{K}.\\
See \cite{B} for a comprehensive proof.
\end{proof}

\section{How to compute spanning trees}
In this section we discuss some combinatorial properties of the spanning trees of an interesting class of graphs, the Jahangir graphs.
We will show how to identify, how many and which are the spanning trees of such graphs in a more simple way than the calculation of the eigenvalues of the Laplacian matrix associated to them.
\begin{definition}\rm{
The \textit{Jahangir graph} $J_{n,m}$, for $n\geq 2, m \geq 3$, is a graph on $nm + 1$ vertices consisting of a cycle $C_{nm}$ with one additional vertex which is adjacent to $m$ vertices of $C_{nm}$ at distance $n$ to each other on $C_{nm}$.}
\end{definition}
\begin{remark} \rm{
We can also say that the Jahangir graph $J_{n,m}$ consists of a cycle $C_{nm}$ which is further divided into $m$ cycles of equal length $C_{n+2}$ such that each pair has just one common vertex and consecutive pairs have exactly one common edge.
}
\end{remark}
\begin{lemma}[Cycles in $J_{2,m}$]
Let $J_{2,m}$ be the Jahangir graph consisting of $m$ adjacent cycles and let $C_{(k)}$ be such cycles, $1\leq k \leq m$. Let $\tau_m$ denote the global number of cycles of $J_{2,m}$ and $C_{(i_1,i_2,\ldots,i_k)}$ the cycle obtained by joining the consecutive cycles $C_{(i_1)},C_{(i_2)},\ldots,C_{(i_k)}$. Then $\tau_m = m^2-m+1$ and $\mid C_{(i_1,i_2,\ldots,i_k)} \mid = 2(k+1)\,.$
\end{lemma}
\begin{proof} Let's point out the global number of cycles of the Jahangir graph $J_{2,m}$.
First of all, the $m$ consecutive cycles constituting it, $C_{(1)},C_{(2)},\ldots,C_{(m)}$, all of length 4.
The remaining ones can be obtained by deleting the common edges between cycles in every possible way, namely 1 common edge $m$ times, 2 consecutive common edges $m$ times, and so on. Thus we also have:\\
$C_{(1,2)},C_{(2,3)},\ldots,C_{(m-1,m)}, C_{(m,1)}$,\\[1mm]
$C_{(1,2,3)},C_{(2,3,4)},\ldots,C_{(m-1,m,1)}, C_{(m,1,2)},$\\[1mm]
$C_{(1,2,3,4)},C_{(2,3,4,5)},\ldots,C_{(m-1,m,1,2)}, C_{(m,1,2,3)},$\\
$\ldots\ldots\ldots\ldots,$\\
$C_{(1,2,\ldots,m-1)},C_{(2,3,\ldots,m-1,m)},C_{(3,4,\ldots,m,1)},\ldots, C_{(m-1,m,1,\ldots,m-3)}, C_{(m,1,\ldots,m-2)},$\\[1mm]
$C_{(1,2,3,\ldots,m)}.$\\[2mm]
Summarizing, the cycles of the graph $J_{2,m}$ are of the type $C_{(i_1,i_2,\ldots,i_k)}$, where $i_{j} \in \{1,2,\ldots,m\}$ and $1\leq k \leq m$, such that $i_{j+1}= i_{j}+1$\, if \, $i_{j}\neq m$\, and \, $i_{j+1}=1$\, if \, $i_{j}=m.$\\
So for each $k\leq m-1$, it results that the total number of cycles $C_{(i_1,i_2,\ldots,i_k)}$ is $m$. Hence the global number of cycles in $J_{2,m}$ is just $m(m-1)+1$.\\
In addition, it is clear from the above construction that $C_{(i_1,i_2,\ldots,i_k)}$ is obtained by deleting $k-1$ common edges between the adjacent cycles $C_{(i_1)},C_{(i_2)},$ $\ldots,C_{(i_k)}$.
Therefore the order of cycles $C_{(i_1,i_2,\ldots,i_k)}$ can be determined by adding orders of all $C_{(i_1)},C_{(i_2)},\ldots,C_{(i_k)}$ and subtracting $2(k-1)$ from the sum, since the common edges are being counted twice; this implies that $\mid C_{(i_1,i_2,\ldots,i_k)} \mid= \sum_{t=1}^k\mid C_{(i_t)}\mid -2(k-1)=2(k+1).$
\end{proof}

\noindent Now we intend to present an algorithm for enumerating and illustrating explicitly all the spanning trees of a Jahangir graph $J_{n,m}$.
\begin{remark} \rm{
The number of common edges among the cycles $C_{n+2}$ of the graph $J_{n,m}$ that a spanning tree can present is a positive integer not greater than $m$. There cannot exist spanning trees without any common edge because, being the only edges of the graph connected to the central vertex, this would be isolated.
}
\end{remark}
\noindent So, the problem of determining all the spanning trees of a Jahangir graph $J_{n,m}$ can be decomposed into subproblems through a classification of the spanning trees on the ground of the number of common edges between the cycles $C_{n+2}$ they have.\\
Let $n,m \in \NN $ be fixed, $n\geq2$ and $m\geq 3$.\\
Let's decompose the problem to compute $\sigma(J_{n,m})$ by calculating the spanning trees with the same number $k\in \NN$ of common edges, $1\leq k\leq m$.\\
It is possible to arrange $k$ common edges of $J_{n,m}$ in $\alpha=$ $m \choose k$ distinct manners, but different types of spanning trees could be generated for 
the choice of $k$ edges from the $m$ edges that are common to two cycles. To this end we must classify the $\alpha$ sets of $k$ indices in equivalence classes depending on the sequential structure of the common edges and count such classes.\\
Let's present the instructions to compute $\sigma(J_{n,m})$.\\[1mm]
Assign $n,m,k$ and consider $\alpha=$ $m \choose k$.\\
Build the $\alpha\! \times\! k$ matrix of the combinations of $m$ positive integers in sets of $k$
\[
B=\left( \begin{array}{cccccc}
1 & 2 & \dots & k-2 & k-1 & k\\
\vdots & & & & & \\
1 & m-k+2 & \dots & m-2 & m-1 & m\\
2 & 3 & \dots & k-1 & k & k+1\\
\vdots & & & & & \\
2 & m-k+2 & \dots & m-2 & m-1 & m\\
\vdots & & & & & \\
m-k+1 & m-k+2 & \dots & m-2 & m-1 & m
\end{array} \right) \,.
\]
The problem moves in examining what rows of the matrix are equivalent to each other in the above sense and how many groups of equivalent rows exist.\\
In particular, the rows $(1 \;\; 2 \, \dots \, k), \; (2 \;\; 3 \, \dots \, k\!+\!1), \; (3 \;\; 4 \, \dots \, k\!+\!2),\, \dots,\, $ $(m\!-\!k\!+\!1 \;\; m\!-\!k\!+\!2 \, \dots \, m)$ 
are equivalent, then they represent an equivalence class of spanning trees with $k$ common edges.\\
Starting from $B$, for finding all the equivalence classes it is possible to generate another $\alpha\! \times\! k$ matrix $C$ which take in account the mutual dispositions of the $k$ common edges.\\
The entries in $C$ are determined as follows: transform any row $B[i]=(b_{i1} \;\; b_{i2} \dots b_{ik})$ of $B$ into the row $C[i]=(\gamma_{i1} \;\; \gamma_{i2} \dots \gamma_{ik})$ of $C, \, i=1,\dots,\alpha$, where $\gamma_{ij}= b_{i,\,j+1}-b_{ij}-1,\, \forall \,j=1,\dots, k\!-\!1$\,;\, $\gamma_{ik}=b_{i1}-b_{ik}-1+m$ and, in addition, make a permutation of the entries of $C[i]$ in increasing order, so that lastly we will indicate such entries as $c_{i1}, c_{i2}, \dots, c_{ik}$, respectively.\\
Therefore the following relation is established:\vspace{1mm}\par
rows in $B$ are equivalent $\Longleftrightarrow$ the corresponding rows in $C$ are equal.\\[1mm]
Consequently, the number of equivalence classes of the spanning trees with $k$ common edges in $J_{n,m}$ is the number of distinct rows of $C$.\\
Let $h$ be the number of distinct rows in $C$ and suppose that the row $C[i]=(c_{i1} \;\; c_{i2} \dots c_{ik})$ of $C$ repeats itself $x_t$ times, for $t=1,\dots, h$.\\
The contribution of the spanning trees related to $C[i]$ is:\vspace{1mm}\par
\hspace{2,8cm} $s_t=n^k\,(c_{i1}+1)(c_{i2}+1)\cdots(c_{ik}+1)$.\\[1mm]
Finally, the number of the spanning trees of $J_{n,m}$ with $k$ common edges is:\vspace{1mm}\par
\hspace{3,8cm} $\sigma(J_{n,m})_k=\sum_{t=1}^h s_t\,x_t$.\\[1mm]
By iterating the process for $k=1,\dots, m$, we obtain\, $\sigma(J_{n,m})$.
\begin{example}\label{ex1}\rm{
Let's apply the above algorithm to the Jahangir graph $J_{2,4}$\,.
\vspace{-6mm}\par\noindent
\begin{figure}[htbp]
\begin{center}
\includegraphics[width=5cm, height=2.5cm]{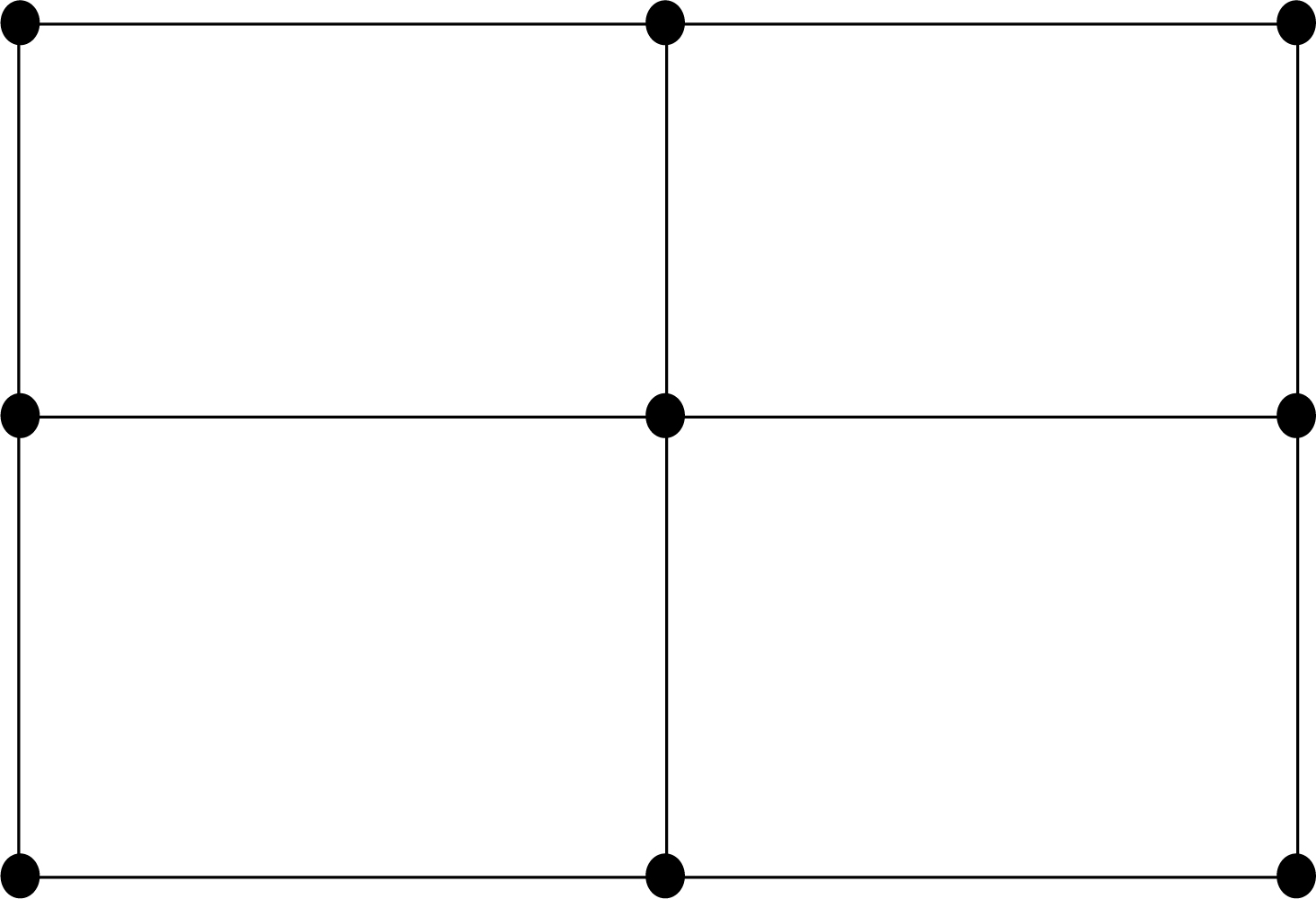}
\caption{Graph $J_{2,4}$}
\end{center}
\end{figure}
\vspace{-6mm}\par\noindent
We have to determine $\sigma(J_{2,4})_{k}$, for $k=1,2,3,4$.\\[1mm]
Let $k=1$.\\
By the cutting-down method, it easily results: \,$4\,(2\cdot 4)=32$\,.\\[1mm]
Let $k=2$.\\
Let $B$ be the $4 \choose 2$ $\!\!\times 2$ matrix \,$\left( \begin{array}{cc}
1 & 2 \\
1 & 3 \\
1 & 4 \\
2 & 3 \\
2 & 4 \\
3 & 4
\end{array} \right)$; the matrix $C$ will be \,$\left( \begin{array}{cc}
0 & 2 \\
1 & 1 \\
0 & 2 \\
0 & 2 \\
1 & 1 \\
0 & 2
\end{array} \right)$\,.\\
The number of equivalence classes of the spanning trees with two common edges in $J_{2,4}$ is 2, that is the number of distinct rows of $C$. In particular, the row $C[1]=(0 \;\; 2)$ repeats itself four times, the row $C[2]=(1 \;\; 1)$ repeats itself twice.\\
The contribution of the spanning trees related to $C[1]$ is \, $s_1=2^2\cdot 1\cdot 3=12$\,; \, that one related to $C[2]$ is \, $s_2=2^2\cdot 2\cdot 2=16$\,. Therefore the number of the spanning trees of $J_{2,4}$ with two common edges is\, $\sigma(J_{2,4})_2=4\cdot 12+2\cdot 16=80$.\\[1mm]
Let $k=3$.\\
Let $B$ be the $4 \choose 3$ $\!\!\times 2$ matrix $\left( \begin{array}{ccc}
1 & 2 & 3\\
1 & 2 & 4\\
1 & 3 & 4\\
2 & 3 & 4
\end{array} \right)$; the matrix $C$ will be $\left( \begin{array}{ccc}
0 & 0 & 1\\
0 & 0 & 1\\
0 & 0 & 1\\
0 & 0 & 1
\end{array} \right)$\,.\\
The 4 rows of $C$ are the same, so the contribution of the spanning trees is \,$2^3\cdot 1\cdot 1\cdot 2=16$\, and $\sigma(J_{2,4})_3=4\cdot 16=64$.\\[1mm]
Let $k=4$.\\
This case is trivial; in fact, $B$ is the row matrix \,$(1 \;\; 2 \;\; 3 \;\; 4)$\, and $C$ the row matrix $(0 \;\; 0 \;\; 0 \;\; 0)$. Consequently\, $\sigma(J_{2,4})_4=2^4=16$.\\[2mm]
In conclusion, $\sigma(J_{2,4})=32+80+64+16=192$\,.
}
\end{example}
\noindent The strength of the algorithmic description introduced for the calculation of spanning trees of any Jahangir graph lies in the fact that with simple operations between integers we can model the totality of dispositions which assume the common edges of the inner cycles $C_{n+2}$ of the graph.\\
Thus we have examined an alternative method which, for this type of graphs, qualitatively and quantitatively solves the problem of determining the spanning trees. It is significant to observe that, by applying the Matrix Tree Theorem, only the total number of such trees is determined, indeed through this algorithm are also found all the possible graphs which originate from the given graph after applying the cutting-down method.
\vspace{6mm}\par\noindent
{\Large \textbf{The source code of the algorithm}}\\[2.5mm]
Following the Example \ref{ex1}, we execute the algorithm for the computation of $\sigma(J_{n,m})_k$ in programming language $C$\,.\\[3mm]
{\small
$\sharp$include $<$stdio.h$>$\\
$\sharp$include $<$stdlib.h$>$\\
$\sharp$include $<$math.h$>$\vspace{2mm}\par
\noindent int** A;\\
int row = 0;\\
int binomial(int m, int k);\\
int* bubbleSort(int* A, int n);\\
void combination(int arr[\,],int n,int r,int index,int data[\,],int i);\\
void createA(int arr[\,], int n, int r);\vspace{2mm}\par
\noindent int main(\,) \{\par
\hspace{3mm}	int n = 2;\par
\hspace{3mm}	int m = 12;\par
\hspace{3mm}	int k = 7;\vspace{2mm}\par
\hspace{3mm}	int i, j, t;\par
\hspace{3mm}	int flag;\vspace{2mm}\par
\hspace{-1mm}	int positions[m];\par
\hspace{-1mm}    for (i = 0; i $<$ m; i++) positions[i] = i+1;\vspace{2mm}\par
\hspace{-1mm}    int alpha = binomial(m,k);\vspace{2mm}\par
\hspace{-1mm}    A = (int **)malloc(alpha * sizeof(int*));\par
\hspace{3mm}	for(i = 0; i $<$ alpha; i++) A[i] = (int *)malloc(k * sizeof(int));\vspace{2mm}\par
\hspace{3mm}	int **G = (int **)malloc(alpha * sizeof(int*));\par
\hspace{3mm}	for(i = 0; i $<$ alpha; i++) G[i] = (int *)malloc(k * sizeof(int));\vspace{2mm}\par
\hspace{3mm}	int **G\_s = (int **)malloc(alpha * sizeof(int*));\par
\hspace{3mm}	for(i = 0; i $<$ alpha; i++) G\_s[i] = (int *)malloc(k * sizeof(int));\vspace{2mm}\par
\hspace{3mm}	int c[alpha];\vspace{2mm}\par
\hspace{3mm}	unsigned long long int numSpanningTree = 0;\vspace{2mm}\par
\hspace{-1mm}    createA(positions, m, k);\vspace{2mm}\par
\hspace{3mm}	for (i = 0; i $<$ alpha; i++) \{\par
	\hspace{5mm}	for (j = 0; j $<$ k; j++) \{\par
	\hspace{7mm}		if (j != k$-$1) \{\par
	\hspace{9mm}			G[i][j] = A[i][j+1] $-$ A[i][j] $-$ 1;\par
	\hspace{7mm}		\} else \{\par
	\hspace{9mm}			G[i][j] = A[i][0] $-$ A[i][j] + m $-$ 1;\par
	\hspace{7mm}		\}\par
	\hspace{5mm}	\}\par
	\hspace{3mm}\}\vspace{2mm}\par
	\hspace{3mm}for (i = 0; i $<$ alpha; i++) {\par
	\hspace{5mm}	G\_s[i] = bubbleSort(G[i], k);\par
	\}\vspace{2mm}\par
	\hspace{3mm}struct graphStruct \{\par
	\hspace{5mm}	int g[k];\par
	\hspace{5mm}	int beta;\par
	\};\vspace{2mm}\par
	\hspace{3mm}struct graphStruct *graph;\par
	\hspace{3mm}graph = malloc(alpha*sizeof(struct graphStruct));\par
	\hspace{3mm}for (i = 0; i $<$ alpha; i++) graph[i].beta = 0;\vspace{2mm}\par
	\hspace{3mm}for (i = 0; i $<$ alpha; i++) \{\par
	\hspace{5mm}	if (i == 0) \{\par
	\hspace{7mm}		for (j = 0; j $<$ k; j++) \{\par
	\hspace{9mm}			graph[i].g[j] = G\_s[i][j];\par
	\hspace{7mm}		\}\par
	\hspace{7mm}		graph[i].beta++;\par
	\hspace{5mm}	\} else \{\par
	\hspace{7mm}		for (t = 0; t $<$ i; t++) \{\par
	\hspace{9mm}			for (j = 0; j $<$ k; j++) \{\par
	\hspace{11mm}				if (G\_s[i][j] == graph[t].g[j]) \{\par
	\hspace{13mm}					flag = 1;\par
	\hspace{11mm}				\} else \{\par
	\hspace{13mm}					flag = 0;\par
	\hspace{13mm}					break;\par
	\hspace{11mm}				\}\par
	\hspace{9mm}			\}\par
	\hspace{9mm}			if (flag == 1) \{\par
	\hspace{11mm}				graph[t].beta++;\par
	\hspace{11mm}				break;\par
	\hspace{9mm}			\}\par
	\hspace{7mm}		\}\par
	\hspace{7mm}		if (flag == 0) \{\par
	\hspace{9mm}			for (j = 0; j $<$ k; j++) \{\par
	\hspace{11mm}				graph[i].g[j] = G\_s[i][j];\par
	\hspace{9mm}			\}\par
	\hspace{9mm}			graph[i].beta++;\par
	\hspace{7mm}		\}\par
	\hspace{5mm}	\}\par
	\}\vspace{2mm}\par
\hspace{3mm}	for(i = 0; i $<$ alpha; i++) \{\par
	\hspace{5mm}	if (graph[i].beta != 0) \{\par
	\hspace{7mm}		c[i] = pow(n, k);\par
	\hspace{7mm}		for (j = 0; j $<$ k; j++) \{\par
	\hspace{9mm}			c[i] *= graph[i].g[j] + 1;\par
	\hspace{7mm}		\}\par
	\hspace{5mm}	\}\par
	\}\vspace{2mm}\par
\hspace{3mm}	for(i = 0; i $<$ alpha; i++) \{\par
	\hspace{5mm}	if (graph[i].beta != 0) \{\par
	\hspace{7mm}		numSpanningTree += c[i]*graph[i].beta;\par
	\hspace{5mm}	\}\par
	\}\vspace{2mm}\par
\hspace{3mm}	printf(''SPANNING TREES of J(\%d, \%d) with k = \%d common edges: \%llu$\backslash$n'', n, m, k, numSpanningTree);\par
\hspace{3mm}    return 0;\\
\}\vspace{2mm}\par
\noindent int binomial(int n, int k) \{\par
 \hspace{-1mm}  if (k==0 $\mid \, \mid$ k==n) return 1;\par
 \hspace{-1mm}  return  binomial(n$-$1, k$-$1) + binomial(n$-$1, k);\\
\}\vspace{2mm}\par
\noindent int* bubbleSort(int* A, int n) \{\par
\hspace{3mm}	int i, j, temp;\par
\hspace{3mm}	for (i = 0; i $<$ n$-$1; i++) \{\par
	\hspace{5mm}	for (j = 0; j $<$ n$-$1; j++) \{\par
	\hspace{7mm}		if(A[j] $>$ A[j+1]) \{\par
	\hspace{9mm}			temp = A[j+1];\par
	\hspace{9mm}			A[j+1] = A[j];\par
	\hspace{9mm}			A[j] = temp;\par
	\hspace{7mm}		\}\par
	\hspace{5mm}	\}\par
\hspace{3mm}	\}\par
\hspace{3mm}	return A;\\
\}\vspace{2mm}\par
\noindent void combination(int arr[\,], int n, int r, int index, int data[\,], int i) \{\par
\hspace{-1mm}    if (index == r) \{\par
    \hspace{3mm}	int j;\par
    \hspace{3mm}    for (j=0; j $<$ r; j++) \{\par
    \hspace{5mm}    	A[row][j] = data[j];\par
	\hspace{5mm}	\}\par
	\hspace{5mm}	row++;\par
    \hspace{3mm}    return;\par
\hspace{-1mm}    \}\vspace{2mm}\par
\hspace{-1mm}    if (i $>$= n) return;\vspace{2mm}\par
\hspace{-1mm}    data[index] = arr[i];\par
\hspace{-1mm}    combination(arr, n, r, index+1, data, i+1);\par
\hspace{-1mm}    combination(arr, n, r, index, data, i+1);\\
\}\vspace{2mm}\par
\noindent void createA(int arr[\,], int n, int r) \{\par
\hspace{-1mm}    int data[r];\par
\hspace{-1mm}    combination(arr, n, r, 0, data, 0);\\
\}
}
\smallskip
\section{An application on secure data transmission}
{\normalsize Commutative and computational algebra, algebraic graph theory and
combinatorics could help applied science to give solutions on
real problems concerning different fields in order to translate or evaluate theoretic results into concrete realities.\\
Algebraic and geometric models are becoming increasingly
significant because they are actively directed to other areas of sciences and
technology, specifically in physical, medical, statistical sectors, engineering, computer science, and so
on. For example, models of this type are substantially given by
graphs and matrices in algebraic systems.\\
Modeling through graphs are principally used to represent processes in which there are
sets with relations between the elements; in particular, graphs
can be employed to analyze connection problems and are basically considered in contexts such as: telecommunication systems, interchange networks, transportation optimal plan of indivisible goods, microwave engineering in resonant structures, coding theory, data organization, flows of computation, research algorithms for the web, security to encrypt messages or to transmit confidential information.
It is also possible to investigate classes of simple graphs in order to
introduce useful procedures for setting suitable solutions of
several examples around the fields of military security and of urban
and territorial analysis.\\
In this section we give a practical application of the
combinatorial properties of the spanning trees of Jahangir graphs
studied along the paper. More precisely, we consider the technique
developed to determine how many and what are the spanning trees of
this class of graphs and explain how the method introduced
for computing the number of the spanning trees, is a good
instrument for optimizing the transmission of confidential information.\\
The application concerns a procedure of data transmission arising from
real security problems. It is necessary to communicate the type of
arming situated inside on some nuclear sites in a country. The
suitable model to represent this situation is a connected graph.
The nuclear sites are located and we can represent them
through the vertex set of the above considered simple connected
graph $G$:
\vspace{-.2mm}
\begin{figure}[htbp]
   \begin{center}
      \includegraphics[width=0.35\textwidth]{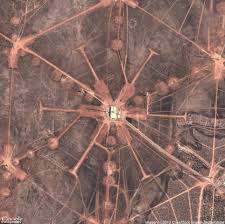}
 \end{center}
\end{figure}
\vspace{-.4cm} \par\noindent We may build all subtrees $G_i$ of
$G$ that contain each vertex of $G$\,. The role
of the supporting graphs $G_i$ is 
fundamental 
in transmitting protected data.
\vspace{-2mm}
\begin{figure}[ht]
   \begin{center}
      \includegraphics[width=0.59\textwidth]{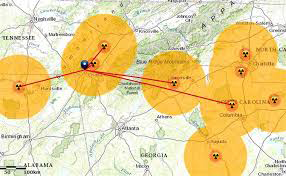}
 \end{center}
\end{figure}
\vspace{-.4cm} \par\noindent The nuclear sites are known and represented
by the finite vertex set of $G$. The arming is classified
through the vertex set of unknown graphs $G_i$. The
message to be sent is the graph $G$ and each
$G_i$ contains the real meaning of the message because its
edges give the connection between the nuclear sites and their
arming. The receiver will get acquainted with the type of arming
placed in every site applying the algebraic procedure to build $G_i$\,.\\
In our case, $G$ is the Jahangir graph $J_{n,m}$. We
associate to $G$ the subgraphs $G_i$, that are the spanning trees of it,
computed through the described method. Hence,
we can represent the nuclear sites through the vertex set of the
connected graph $J_{n,m}$. The arming in every nuclear site can be
classified through the vertex set of another undisclosed graph
$G_i$. In fact the transmitter sends in his message the
drawing of the graph $J_{n,m}$. The receiver will elaborate the
potential information contained in it computing the spanning trees
$G_i$\,.\\
For instance, for $J_{2,3}$ the computation gives the
following spanning trees
\begin{figure}[htbp]
   \begin{center}
      \includegraphics[width=0.66\textwidth]{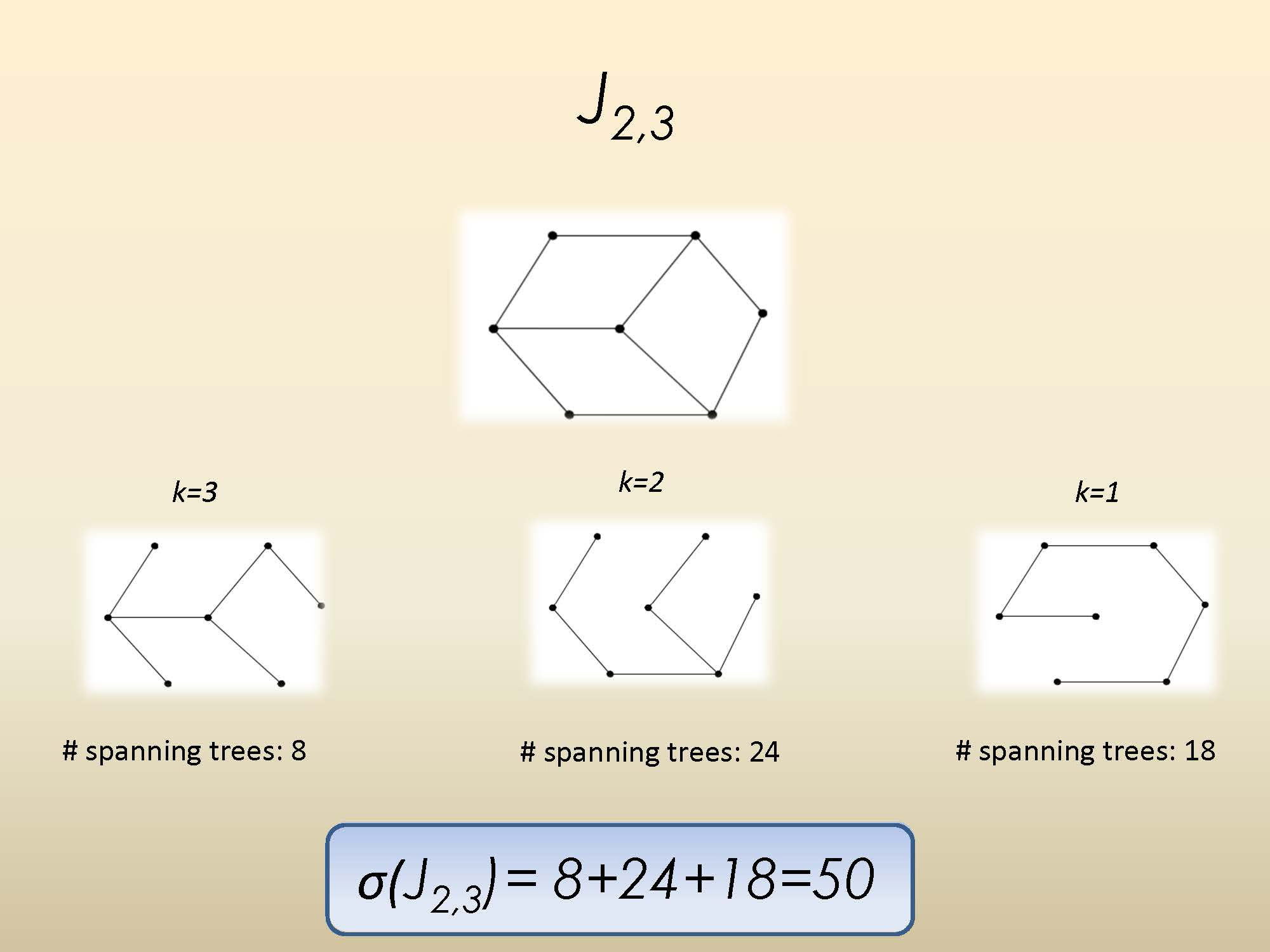}
 \end{center}
\end{figure}
\vspace{-.4cm} \par\noindent The graphs $G_i$, varying the number of common edges $k$ in them, contain the real
meaning of the message because their edges give the connection between the nuclear sites and their arming.
\begin{figure}[htbp]
   \begin{center}
      \includegraphics[width=0.70\textwidth]{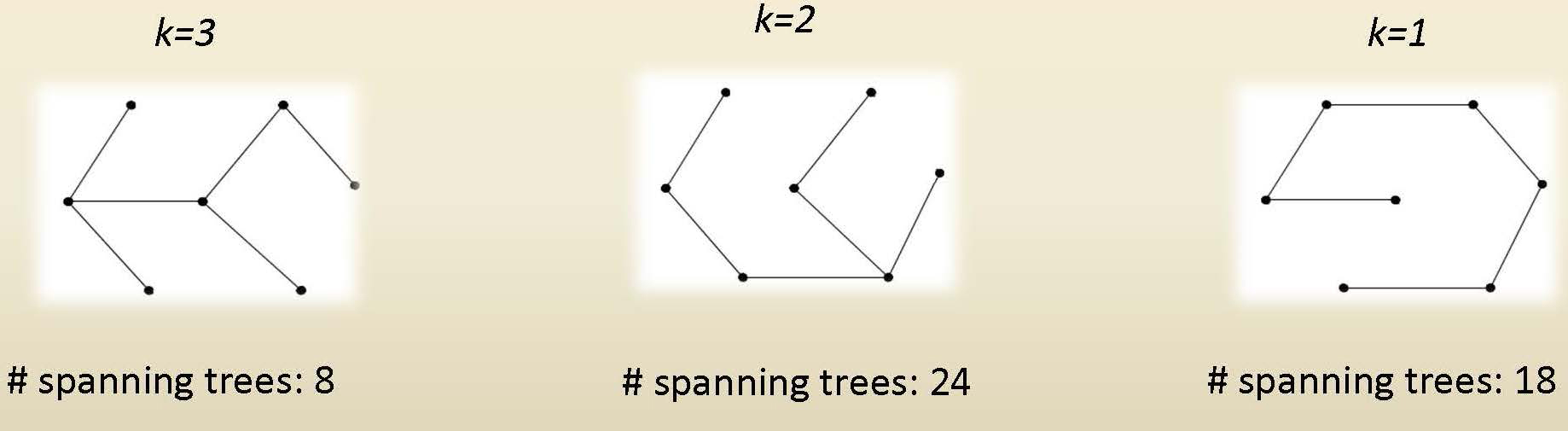}
 \end{center}
\end{figure}
\vspace{-.4cm} \par\noindent So the receiver will get acquainted with the type of arming placed
in each site.

\smallskip
\section{Structure relationships}
Let's study another aspect of the class of Jahangir graphs $ J_{n,m}$ in relation to the calculation of the spanning trees of them.\\
We will want to locate the general structure of the class of Jahangir graphs. Recall that in any Jahangir graph $J_{n,m}$ we can intend the index $n$ as the number of edges that each inner cycle of the graph does not share with all the other ones, and the index $m$ as the number of inner cycles that make up the graph.\\
Considering the totality of Jahangir graphs $J_{n,m}$, it is interesting to study those which have either the same index $n$ or the same index $m$.\\
Taken for example the Jahangir graphs for which $m=3$ and calculating the spanning trees on the ground of the number of common edges which they present, we can write:\vspace{1.5mm}\par
$\sigma(J_{n,3}) = \sigma(J_{n,3})_3 + \sigma(J_{n,3})_2 + \sigma(J_{n,3})_1 = n^3+ 3(2n)n^{3-2} + 3(3n)$
.\\[1.5mm]
Now introduce a statement that shows the relationship between the spanning trees related to two Jahangir graphs whose first indices are consecutive numbers.
\begin{theorem} Let $n, m>2$. Then \,$\displaystyle{\lim_{n\to \infty} \frac{\sigma(J_{n,m})}{\sigma(J_{n-1,m})}=1}$.
\end{theorem}
\begin{proof}
For a fixed positive integer $m >2$ it results that:
$$\sigma(J_{n,m})=\sum_{k=1}^m \sigma(J_{n,m})_k =n^m+ \sum_{i=1}^{m-1} c_i n^i, \, \mathrm{for\,\,} c_i \in \RR\,.$$
Consequently \,$\displaystyle{\lim_{n\to \infty} \frac{\sigma(J_{n,m})}{\sigma(J_{n-1,m})}=\lim_{n\to \infty} \frac{n^m+\sum_{i=1}^{m-1} c_i n^i}{(n-1)^m+ \sum_{i=1}^{m-1} c_i (n-1)^i}=1.}$
\end{proof}

\noindent Observe that the theorem holds even for two non-consecutive values of the first indices.\\
The previous result emphasizes the geometric and analytical aspects of the problem: in fact, it appears that the number of spanning trees of Jahangir graphs having the same second index and the first one tending to infinity, tends to be a constant.\\
It may be noted, explaining the calculation, that the previous function is decreasing: geometrically this means that the number of spanning trees of Jahangir graphs $J_{n,m}$\,, with $m$ constant and $n$ large enough, turns out to be equal; in other words, fixing the number of edges connected to the center, it is as if we were tending a polygon to a circumference by increasing its edges.\\
This leads to the conclusion that, for $n$ sufficiently large, $\sigma(J_{n,m})$ and $\sigma(J_{n-1,m})$ tend to assume values very close to each other.\\
On the other hand, doing a similar study on the increasing of the second index of a Jahangir graph $J_{n,m}$ fixed the first one, it results that the ratio between the number of spanning trees related to two Jahangir graphs whose second indices are consecutive, for $m$ tending to infinity, seems to approximate to a constant, different for each choice of $n$.\\
First, let's indicate the behaviour of the spanning trees of the Jahangir graphs $J_{2,m}$ and $J_{3,m}$, for $m\geq 3$.\\[-2mm]\\
\begin{center}
\begin{tabular}{|l|l|l|}
  \hline
$m$ & $\sigma(J_{2,m})$ & $\sigma(J_{3,m})$  \\
  \hline
3 & 50           & 108             \\
4 & 192          & 525             \\
5 & 722          & 2523            \\
6 & 2700         & 12096           \\
7 & 10082        & 57963           \\
8 & 37632        & 277725          \\
9 & 140450       & 1330668         \\
10& 524172       & 6375621         \\
11& 1956242      & 30547443        \\
12& 7300800      & 146361600       \\
13& 27246962     & 701260563       \\
14& 101687052    & 3359941221      \\
15& 379501250    & 16098445548     \\
16& 1416317952   & 77132286525     \\
\dots & \dots\dots\dots & \dots\dots\dots \\
& & \\

\hline
\end{tabular}
\end{center}
\vspace{2mm}
\noindent
Known such information, now we wonder if it is possible to compute $\sigma(J_{2,m})$ or $\sigma(J_{3,m})$, for any $m$, namely in general if it is possible to identify a relationship among the number of spanning trees of $\sigma(J_{n,m})$, when $n\geq 2$ is assigned.
To this end we introduce the sequences:\vspace{2mm}\par
\hspace{2cm} $\displaystyle{a_{2,m}= \frac{\sigma(J_{2,m+1})}{\sigma(J_{2,m})}},$\, for $m\in \NN, \, m\geq3\,,$\vspace{1mm}\par
\hspace{2cm} $\displaystyle{a_{3,m}= \frac{\sigma(J_{3,m+1})}{\sigma(J_{3,m})}},$\, for $m\in \NN, \, m\geq3\,.$ \vspace{2mm}
\par\noindent
Building sequences of the mentioned type, for any fixed integer $n\geq 2$, that is, considering a countable infinity of sequences, it does highlight a symmetry that can already be grasped visually. Formalizing a relation between these numbers, let's analyze the way in which the problem numerically evolves.\\
The values of the first terms of the two sequences introduced above, for all $m$, are the following:\\[-2mm]
\begin{center}
\begin{tabular}{|l|l|l|}
  \hline
$m$ & $a_{2,m}$ & $a_{3,m}$\\
  \hline
3 & 3.84 & 4.86 \\
4 & 3.7604 & 4.8057 \\
5 & 3.7396 & 4.7943 \\
6 & 3.7340 & 4.7919 \\
7 & 3.732593 & 4.791418 \\
8 & 3.732196 & 4.791315 \\
9 & 3.7320897 & 4.7912935 \\
10& 3.7320612 & 4.7912890 \\
11& 3.732053600 & 4.7912880 \\
12& 3.732051556 & 4.791287899 \\
13& 3.732051008 & 4.791287858 \\
14& 3.732050861 & 4.7912878497 \\
15& 3.732050822 & 4.7912878479 \\
\dots & \dots\dots\dots & \dots\dots\dots \\
& & \\
\hline
\end{tabular}
\end{center}
\vspace{2mm}
\noindent
It is evident that, for any fixed $n \geq 2$ and for any variation of the second index $m$, the number of spanning trees has a linear increase.
\begin{proposition}
When $n=2$, the sequence $\,\displaystyle{a_{2,m}= \frac{\sigma(J_{2,m+1})}{\sigma(J_{2,m})}}$,\, for $m\to +\infty$, tends to $\,\delta_2=2\!+\!\sqrt{3}$.
\end{proposition}
\begin{proof}
We point out that the integer sequence $\sigma(J_{2,m})$ is given by the formula $\,-((2\!+\!\sqrt{3})^m-1)\,((2\!-\!\sqrt{3})^m-1)$\,, hence\\[2mm]
\begin{tabular}{ll}
$\displaystyle{\lim_{m\to +\infty} a_{2,m}}\!\!$ &$= \displaystyle{\lim_{m\rightarrow +\infty} \frac{((2\!+\!\sqrt{3})^{m+1}-1)\,((2\!-\!\sqrt{3})^{m+1}-1)}{((2\!+\!\sqrt{3})^m-1)\,((2\!-\!\sqrt{3})^m-1)}}$\\
&$= \displaystyle{\lim_{m\to +\infty} \frac{(2\!+\!\sqrt{3})^{m+1}-1}{(2\!+\!\sqrt{3})^m-1}}$\\
&$= \displaystyle{\lim_{m\to +\infty} \frac{(2\!+\!\sqrt{3})^{m+1}}{(2\!+\!\sqrt{3})^m-1}}=2\!+\!\sqrt{3}$\,.
\end{tabular}

\vspace{-5mm}
\end{proof}
\vspace{1mm}
\noindent
We note that when $n=3$, even taking into account the previous table, for the sequence $\,\displaystyle{a_{3,m}= \frac{\sigma(J_{3,m+1})}{\sigma(J_{3,m})}}$,\, it results \,
$\displaystyle{\lim_{m\to +\infty} a_{3,m}}=\delta_{3}\approx 4.791287... \; . $\\[2mm]
And so on, for any increase of $n$.\\[2mm]
More generally, we can state the following \vspace{-.5mm}
\begin{theorem}
Let $n,m \in \NN$ and $n\geq 2$ fixed. Let $a_{n,m}= \displaystyle{\frac{\sigma(J_{n,m+1})}{\sigma(J_{n,m})}}\,.$ Then \par
\hspace{3cm} $\displaystyle{\lim_{m\to +\infty} a_{n,m}}=\delta_{n}$, \, where\, $\delta_{n}\in \RR$\,.\\[2mm]
Moreover, for $m>3$, it is:\par
\hspace{3,7cm} $\sigma(J_{n,m})>(\delta_{n})^{m-3}\,\sigma(J_{n,3})\,,$\\[2mm]
and, when $m$ tends to infinity, \,$\sigma(J_{n,m})$ approximates \,$(\delta_{n})^{m-3}\,\sigma(J_{n,3})\,.\quad \Box$
\end{theorem}

\vspace{5mm}\par\noindent
{\Large \textbf{Acknowledgements}}\\[2.5mm]
The research that led to the present paper was partially supported by a grant of the group GNSAGA of INdAM, Italy.
\smallskip

}
\end{document}